\documentclass[12pt,a4paper]{amsart}
\usepackage{url}
\usepackage{amsmath,amsthm,amsfonts,amssymb,latexsym}

\newtheorem{theorem}{Theorem}[section]

\newtheorem{claim}[theorem]{Claim}

\theoremstyle{definition}

\newtheorem{question}[theorem]{Question}

\newcommand{\U}{\mathcal U}
\newcommand{\w}{\omega}

\newcommand{\IP}{\mathbb P}

\newcommand{\B}{\mathcal{B}}

\newcommand{\K}{\mathcal{K}}
\newcommand{\A}{\mathcal{A}}

\newcommand{\F}{\mathcal{F}}

\newcommand{\V}{\mathcal{V}}

\newcommand{\bb}{\mathfrak b}

\newcommand{\uhr}{\upharpoonright}

\newcommand{\la}{\langle}
\newcommand{\ra}{\rangle}

\newcommand{\hot}{\mathfrak}

\newcommand{\nothing}[1]{}

\newcommand{\zrost}{\w^{\uparrow\w}}

\title[Covering properties of $\w$-mad families]{Covering properties of $\w$-mad families}

\author{Leandro Aurichi  and Lyubomyr Zdomskyy}

\address{Instituto de Ci\^encias Matem\'aticas e de Computa\c c\~ao, Universidade de
 S\~ao Paulo (ICMC-USP), S\~ao Carlos, Brazil}
\email{aurichi@icmc.usp.br}
 \urladdr{icmc.usp.br/pessoas/aurichi/}

\address{Institut f\"ur Diskrete Mathematik und Geometrie, Technische Universit\"at Wien, Wiedner Hauptstra\ss e 8-10/104, 1040 Wien, Austria.}
\email{lzdomsky@gmail.com}
\urladdr{http://dmg.tuwien.ac.at/zdomskyy/}


\thanks{The first author was  supported by
FAPESP (2017/09252-3). The second author would like to thank  the Austrian Science Fund FWF (Grant I 2374-N35) for generous support for this research. }

\begin{document}
\begin{abstract}
We prove that CH implies the existence of a
Cohen-indestructible mad family such that the Mathias forcing
associated to its filter adds dominating reals, while $\hot b=\hot
c$ is consistent with the negation of this statement as witnessed by
the Laver model for the consistency of Borel's conjecture.
\end{abstract}

\maketitle

\section{Introduction}
Recall that an infinite $\A\subset[\w]^\w$ is called a \emph{mad
family}, if $|A_0\cap A_1|<\w$ for any distinct $A_0,A_1\in \A$, and for every $B\in[\w]^\w\setminus\A$
there exists $A\in\A$ such that $|B\cap A|=\w$.
In \cite[Theorem~2.1]{Bre98} Brendle constructed under CH a mad
family $\A$ on $\w$ such that the Mathias forcing\footnote{Since we
shall not analyze this poset directly but rather use certain
topological characterizations, we refer the reader to, e.g.,
\cite{Bre98} for its definition.} $M_{\F(\A)}$ associated to the
filter
$$ \F(\A)=\{F\subset\w:\exists \A'\in[\A]^{<\w}(\w\setminus\cup\A'\subset^* F)\}$$
adds a dominating real. In the same paper Brendle asked whether such
a mad family can be constructed outright in ZFC. This question has
been answered in the affirmative in \cite{GuzHruMar??} and later
independently also in \cite{ChoRepZdo15} using different methods. Since the goal
  of these studies was to
find forcings destroying a given mad family while keeping
(certain subsets of) the ground model reals unbounded (and perhaps
having other useful properties), this motivates the following version of
Brendle's question: Suppose that a mad family $\A$ cannot be
destroyed by some very ``mild''  forcing $\IP$, i.e., it remains
maximal in $V^{\IP}$, must then $M_{\F(\A)}$ add dominating reals?
This approach seems natural because if $\A$ is already destroyed by
$\IP$, there is no need to use its Mathias forcing for its
destruction in a hypothetic construction of a model where, e.g.,
$\hot b$ should stay small. $\hot b$ as well as other notions used
in the introduction will be defined in the next section. In this
note we consider this question for $\IP$ being the Cohen forcing
$\mathbb C$. Mad families $\A$ which remain maximal in $V^{\mathbb
C}$ will be called Cohen-indestructible.

\begin{theorem} \label{p_equals_c}
$\hot p=\mathit{cov}(\mathcal N)=\hot c$ implies the existence of a
Cohen-indestructible mad family $\A$ such that $M_{\F(\A)}$ adds a
dominating real.
\end{theorem}

Recall from \cite{Kur01} that a mad family $\A$ is called
\emph{$\w$-mad} if for every sequence $\la X_n:n\in\w\ra$ of
elements of $\F(\A)^+$ there exists $A\in\A$ such that $|A\cap
X_n|=\w$ for all $n$. Cohen-indestructible mad families are closely
related to $\w$-mad ones, see \cite{Mal90} or
\cite[Theorem~4]{Kur01}: Every $\w$-mad family is
Cohen-indestructible, and if $\A$ is Cohen-indestructible, then for
every $X\in\F(\A)^+$ there exists $Y\subset X, $ $Y\in\F(\A)^+$,
such that $\A\uhr Y=\{A\cap Y:A\in\A, A\cap Y$ is infinite$\}$ is
$\w$-mad as a mad family on $Y$.

In the proof of Theorem~\ref{p_equals_c} we actually construct an
$\w$-mad family. The next theorem shows that $\hot b=\hot c$ would
not suffice in Theorem~\ref{p_equals_c} (recall that $\hot p\leq\hot
b$, see, e.g., \cite{Bla10}), and we do not know whether any of the
equalities $\hot p=\hot c$ or $\mathit{cov}(\mathcal N)=\hot c$
would be sufficient.

\begin{theorem} \label{laver}
In the Laver model for the consistency of the Borel conjecture, for
every $\w$-mad family $\A$ the poset $M_{\F(\A)}$ does not add
dominating reals. In particular, if $\A$ is Cohen-indestructible,
then there exists $X\in\F(\A)^+$ such that $M_{\F(\A)\uhr X}$ does
not add dominating reals, where $\F(\A)\uhr X$ denotes the filter on
$\w$ generated by the centered family $\{F\cap X:F\in\F(\A)\}$.
\end{theorem}

In our proofs of Theorems~\ref{p_equals_c} and \ref{laver} we shall
not work with the Mathias forcing directly, but rather use the
following characterization obtained in \cite{ChoRepZdo15}: For a
filter $\F$ on $\w$ the poset $M_{\F}$ adds no dominating reals iff
$\F$ has the Menger covering property when considered with the
topology inherited from $\mathcal P(\w)$, which is identified with
the Cantor space $2^\w$ via characteristic functions. Recall from
 \cite{Hur25} that a topological space
$X$ is said to have the Menger property if for every sequence $\la
\U_n : n\in\omega\ra$ of open covers of $X$ there exists a sequence
$\la \V_n : n\in\omega \ra$ such that each $\V_n$ is a finite
subfamily of $\U_n$ and the collection $\{\cup \V_n:n\in\omega\}$ is
a cover of $X$. The current name (the Menger property) has been
adopted because Hurewicz proved in   \cite{Hur25} that for
metrizable spaces his property is equivalent to a certain basis
property considered by Menger in \cite{Men24}. If in the definition
above we additionally require that $\{\cup\V_n:n\in\w\}$ is a
\emph{$\gamma$-cover} of $X$ (this means that the set
$\{n\in\w:x\not\in\cup\V_n\}$ is finite for each $x\in X$), then we
obtain the definition of the Hurewicz covering property introduced
in \cite{Hur27}. These properties are related as follows:

\centerline{$\sigma$-compact \ $\to$ \ Hurewicz \ $\to$ \ Menger \ $\to$ \ Lindel\"of'}

\noindent Contrary to a conjecture of Hurewicz, the class of
metrizable spaces having the Hurewicz property turned out  to be
wider than the class of $\sigma$-compact spaces
\cite[Theorem~5.1]{COC2}. Also, there are ZFC examples of
non-Hurewicz subspaces $X$ of the real line whose all finite powers
are Menger, see \cite{ChaPol??} or \cite{TsaZdo08}.

In light of Theorem~\ref{laver} we would like to ask whether it is
 consistent that $\F(\A)$ is Hurewicz for any $\w$-mad family $\A$.
However, since it is unknown whether $\w$-mad families exist in ZFC, we suggest the following

\begin{question}
Is it consistent that there exist $\w$-mad families and $\F(\A)$ is Hurewicz for any such a family $\A$?
Is this the case in the Laver model?
\end{question}

\section{Proofs}

Let us first recall the definitions of cardinal characteristics
appearing in this paper. $\hot p$ is the minimal cardinality of a
family $\mathcal X\subset [\w]^\w$ such that $\cap\mathcal X'\in
[\w]^\w$ for any $\mathcal X'\in [\mathcal X]^{<\w}$, but there is
no $Y\in [\w]^\w$ such that $Y\subset^*X$ for all $X\in\mathcal X$.
$\hot b$ is the minimal cardinality of an unbounded  subset $B$ of
$\w^\w$ with respect to the following pre-order: $x\leq^* y$ iff
$\{n\in\w: x(n)>y(n)\}$ is finite. Finally, $\mathit{cov}(\mathcal
N)$ is the minimal cardinality of a cover of  $\mathbb R$ by
Lebesgue null sets. It is well-known that $\hot p
=\mathit{cov}(\mathcal N)=\w_1\leq \hot b=\w_2=\hot c$ in the Laver
model, see, e.g., \cite[p.~480]{Bla10} and references therein.

We shall first prove Theorem~\ref{p_equals_c}.  Here we shall often
use the following easy fact without mentioning it: For any countable
collection $\A$ of countable sets, for every $A\in\A$ there exists
$B(A)\in [A]^\w$ such that $B(A)\cap B(A')=\emptyset$ for any
distinct $A,A'\in\A$.
\medskip

\noindent\textit{Proof of Theorem~\ref{p_equals_c}.}  \ We shall
first present the proof under CH,
 and then indicate
what should be changed  to make the proof work under $\hot
p=\mathit{cov}(\mathcal N)=\hot c$.

 Let $\la I_n:n\in\w\ra$ be a sequence of infinite mutually
disjoint subsets of $\w$. For every $k\in\w$ set
$P_k=2^{k+1}\setminus 2^k$ and note that elements of
$\{P_k:k\in\w\}$ are mutually disjoint. Let $\{\la
X^\alpha_n:n\in\w\ra :\alpha<\w_1\}$ be the family of all sequences
of infinite subsets of $\w$. Let us also fix an enumeration
$\{f_\alpha:\alpha<\w_1\}$ of the family $F$ of all increasing
$f\in\w^\w$ such that $\{n\in\w:\exists k(n\in I_k\wedge P_n\subset
f(k))\}$ is infinite. Note that $f\in F$ and $f\leq^* f'$ yields
$f'\in F$, and hence $F$ is obviously dominating with respect to
$\leq^*$.

 By transfinite induction on $\alpha$ we shall construct a sequence
$\la A_\alpha:\alpha<\w_1\ra$ of infinite subsets of $\w$ satisfying
the following properties:
\begin{itemize}
\item[$(i)$] $|A_\beta\cap A_\gamma|<\w$ for all $\beta\neq\gamma$;
\item[$(ii)$] $|A_\beta\cap P_k|\leq 2$ for every $\beta\in\w_1$ and
$k\in\w$;
\item[$(iii)$] For every $a\in [\w_1]^{<\w}$ and $k\in\w$ the set
$\{n\in I_k: \bigcup_{\beta\in a}A_\beta\cap P_n=\emptyset\}$ is
infinite;
\item[$(iv)$] For every $\beta\in\w_1$, if
$|X^\beta_n\setminus\bigcup_{\gamma\in a}A_\gamma|=\w$ for all
$n\in\w$ and finite $a\subset\beta$, then $|A_\beta\cap
X^\beta_n|=\w$ for all $n\in\w$; and
\item[$(v)$] $A_\beta\cap P_k\neq\emptyset$
provided that  $k\in I_n$ and $P_k\subset f_\beta(n) $.
\end{itemize}
Assuming that conditions $(i)$-$(v)$ are satisfied for all
$\beta,\gamma<\alpha$ and $a\subset\alpha$, let us consider the
sequence $\la X^\alpha_n:n\in\w\ra$. Two cases are possible.
\smallskip

1. $|X^\alpha_n\setminus\bigcup_{\gamma\in a}A_\gamma|=\w$ for all
$n\in\w$ and finite $a\subset\alpha$, i.e., the premises of $(iv)$
hold for $\alpha$. Let us note that if we shrink the sets
$X^\alpha_n$'s so that the premises in $(iv)$ are still satisfied,
the required conclusion of the property $(iv)$ becomes harder to
fulfill. Thus passing to an infinite pseudointersection of the
countable family
$$ \{X^\alpha_n\setminus\bigcup_{\gamma\in a}A_\gamma : a\in [\alpha]^{<\w}\} $$
of infinite subsets of $X^\alpha_n$, we may assume that
$|X^\alpha_n\cap A_\beta|<\w$ for all $n\in\w$ and $\beta<\alpha$.
Let $g\in\w^\w$ be such that for all $\beta<\alpha$ there exists
$n\in\w$ with the property $X^\alpha_m\cap A_\beta\subset g(m)$ for
all $m\geq n$. Letting $Y_n=X^\alpha_n\setminus g(n)$, we get that
\begin{itemize}
\item[$(vi)$]
$\bigcup_{n\in\w}Y_n$ is almost disjoint from $A_\beta$ for all
$\beta<\alpha$.
\end{itemize}
\begin{claim}\label{cl01}
For every $m\in\w$ there exists $B_m\in [Y_m]^\w$ such that
$B=\bigcup_{m\in\w}B_n$ has the following properties:
$$ \forall k\in\w\:\forall a\in [\alpha]^{<\w}\big( \{n\in I_k: P_n\cap (B\cup\bigcup_{\beta\in a}A_\beta)=\emptyset\}\mbox{ is infinite}\big) $$
and $|B\cap P_n|\leq 1$ for all $n\in\w$.
\end{claim}
\begin{proof}
 For every $k\in\w$ and $a\in [\alpha]^{<\w}$ set
 $N^k_a=\{n\in I_k: P_n\cap \bigcup_{\beta\in a}A_\beta=\emptyset\}$
 and note that by our assumptions $\{N^k_a:a\in [\alpha]^{<\w}\}$
 is a countable centered family of infinite subsets of
 $I_k$, and hence there exists $N^k\in [I_k]^\w$ such that
 $N^k\subset^*N^k_a$ for all $a$ as above.
Let
$$M_\infty=\{m\in\w:\exists^\infty k\exists n\in N^k (Y_m\cap P_n)\neq\emptyset\}$$
and for every $m\in M_\infty$ set $J_m=\{k\in\w:\exists n\in N^k
(Y_m\cap P_n)\neq\emptyset\}\in [\w]^\w.$ Pick $J'_m\in [J_m]^\w$
for all $m\in M_\infty$ such that $J'_{m_0}\cap J'_{m_1}=\emptyset$
for arbitrary $m_0\neq m_1$ in $M_\infty$. Given $m\in M_\infty$,
for every $k\in J'_m$ pick $n_{m,k}\in N^k$ such that $Y_m\cap
P_{n_{m,k}}\neq\emptyset$, and fix $l_{m,k}\in P_{n_{m,k}}\cap Y_m$.
For every $m\in M_\infty$ set $B_m=\{l_{m,k}:k\in J'_m\}$.

Suppose now that $m\in\w\setminus M_\infty$. Two cases are possible.

$a)$ \ There exists $k_m\in\w $ such that $L_m:=\{n\in
N^{k_m}:Y_m\cap P_n\neq\emptyset\}$ is infinite. Given $k\in\w$, for
every $m$ such that $k=k_m$ find $Q_m\in [L_m]^\w$, and $R_k\in
[N^k]^\w$ such that $Q_{m_0}\cap Q_{m_1}=\emptyset$ for any distinct
$m_0,m_1$ such that $k=k_{m_0}=k_{m_1}$, and $R_k\cap Q_m=\emptyset$ for
all $m$ with $k=k_m$. Now for every $n\in Q_m$ pick $q_{m,n}\in
Y_m\cap P_n$ and set $B_m=\{q_{m,n}:n\in Q_m\}$.

$b)$ \  The set
$$ S_m:= \{\la k,n\ra: k\in\w, n\in N^k, Y_m\cap P_n\neq\emptyset\} $$
is finite. Then let
$$B_m\in \big[Y_m\setminus\bigcup\big\{P_n:\exists k (\la k,n\ra\in S_m)\big\}\big]^\w$$
be such that for each $n$ we have $|B_m\cap P_n|\leq 1$.

Thus we have already constructed the sequence $\la B_m:m\in\w\ra$.
We claim that $B=\bigcup_{m\in\w}B_m$ is as required. By the choice
of $N^k$ it suffices to prove that
$$ \forall k\in\w \big( \{n\in N^k: P_n\cap B=\emptyset\}\mbox{ is infinite}\big). $$
We shall show that if $k=k_m$ for some $m$, then $P_n\cap
B=\emptyset$ for all but maybe one $n\in R_k$. Otherwise $P_n\cap
B=\emptyset$ for all but maybe one $n\in N^k$. Indeed, by the
construction (more precisely, since all $J'_m,$ $m\in M_\infty$ are
mutually disjoint), the union $B_\infty:=\bigcup_{m\in M_\infty}B_m$
has the property that for every $k\in\w$ there exists at most one
$n\in N^k$ such that $B_\infty\cap P_n\neq\emptyset.$ Now if
$m\in\w\setminus M_\infty$ and case $b)$ takes place, then $B_m$
intersects no $P_n$ for $n\in\bigcup_{k\in\w}N^k$. And finally, if
$m\in\w\setminus M_\infty$ and $a)$ takes place with $k=k_m$, then
$B_m\subset\bigcup_{n\in N^k}P_n$ and $B_m\cap\bigcup_{n\in
R_k}P_n=\emptyset$. Since the $I_k$'s (and hence also the $N^k$'s)
are mutually disjoint, this completes our proof.
\end{proof}

\begin{claim}\label{cl02}
Let $\la n_i:i\in\w\ra$ be the increasing enumeration of the
set\footnote{This set is infinite by the definition of $F$ and
$f_\alpha\in F$.} $\{n\in\w:\exists k(n\in I_k\wedge P_n\subset
f_\alpha(k))\}$. Then there exists $C\in[\w]^\w$ such that $|C\cap
A_\beta|<\w$ for all $\beta<\alpha$, $|C\cap P_{n_i}|=1$ for all
$i$, and $C\cap P_n=\emptyset$ if $n\not\in\{n_i:i\in\w\}$.
\end{claim}
\begin{proof}
By $(ii)$ we can find a countable family $G$ of functions in
$\prod_{i\in\w}P_{n_i}$ such that
$A_\beta\cap(\bigcup_{i\in\w}P_{n_i})$ is covered by graphs of at
most 2 elements of $G$, for all $\beta<\alpha$. Now it is easy to
construct  $h\in \prod_{i\in\w}P_{n_i}$ eventually different from
each element of $G$. It follows that $C:=\mathrm{range}(h)$ is as
required.
\end{proof}

Set $A_\alpha=B\cup C$, where $B,C$ are such as in Claims~\ref{cl01}
and \ref{cl02}, respectively. Since $\{n_i:i\in\w\}\cap I_k$ is
finite for all $k\in\w$, it is easy to see that all conditions
$(i)$-$(v)$ are also satisfied for $\beta,\gamma\leq\alpha$ and
$a\in [\alpha+1]^{<\w}$.
\smallskip

2. There exists  $n\in\w$ and a finite $a\subset\alpha$ such that
$X^\alpha_n\subset^* \bigcup_{\gamma\in a}A_\gamma$. Set $A_\alpha=
C$, where $C$ is such as in Claim~\ref{cl02}. Again, all conditions
$(i)$-$(v)$ are  satisfied for $\beta,\gamma\leq\alpha$ and $a\in
[\alpha+1]^{<\w}$.
\smallskip

This completes our construction of a sequence $\la
A_\alpha:\alpha<\w_1\ra$ satisfying $(i)$-$(v)$. By $(i)$ and
$(iv)$, $\A=\{A_\alpha:\alpha<\w_1\}$ is an $\w$-mad family. By
$(iii)$ the family $\U_k=\{O_n:n\in I_k\}$ is an open cover of
$\F(\A)$ for all $k\in\w$, where $O_n=\{X\subset\w:P_n\subset X\}$.
We claim that the sequence $\la\U_k:k\in\w\ra$ witnesses that
$\F(\A)$ is not Menger. Indeed, otherwise there exists $\alpha$ such
that
$$\U:=\{O_n:\exists k\in\w (n\in I_k\wedge P_n\subset f_\alpha(k))\}$$
covers $\F(\A)$. However, $P_n\cap A_\alpha\neq\emptyset$ for all
$n\in I_k$ such that $P_n\subset f_\alpha(k)$ for some $k\in\w$, which
means that $\F(\A)\ni\w\setminus A_\alpha\not\in\cup\U$. This  leads
to a contradiction and thus finishes our proof under CH.
\smallskip

Except for the proof of Claim~\ref{cl02}, we have used CH to produce
at stage $\alpha$ a pseudointersection of a centered family of
infinite subsets of $\w$ of size $|\alpha|$, and $\hot p=\hot c$
suffices for finding such pseudointersections by the definition of
$\hot p$.

Regarding Claim~\ref{cl02},  we shall show\footnote{We believe that
this straightforward argument is well-known, but we were unable to
locate it in the literature.} that for any family
$G\subset\prod_{i\in\w}P_{n_i}$ of size $<\mathit{cov}(\mathcal N)$
there exists $h\in \prod_{i\in\w}P_{n_i}$ eventually different from
all elements of $G$ (here we use the same notation as in the
formulation of Claim~\ref{cl02}). Indeed, let
 $\mu$ be the Borel measure on
$\prod_{i\in\w}P_{n_i}$ such that for every $i\in\w$ and $s\in
\prod_{j\leq i}P_{n_j}$ we have $\mu([s])=\prod_{j\leq i}2^{-n_j}$,
where
$$ [s] = \{x\in \prod_{i\in\w}P_{n_i}:x\uhr(i+1)=s\}.$$
By \cite[Theorem~17.41]{Kec95} the measurable space $\la
\prod_{i\in\w}P_{n_i},\mu\ra$ is isomorphic to $\mathbb R$ equipped
with the standard Lebesgue measure $\lambda$. A simple calculation
shows that
$$ \mu\{x\in \prod_{i\in\w}P_{n_i}:\exists^\infty i\in\w (x(i)=g(i)) \}=0 $$
for every $g\in \prod_{i\in\w}P_{n_i}$. Since (by the definition)
$\mathbb R$ cannot be covered by fewer than $\mathit{cov}(\mathcal
N)$ many null subsets, neither $\la \prod_{i\in\w}P_{n_i},\mu\ra$
can, and hence Claim~\ref{cl02} holds for families $G$ of size
$<\mathit{cov}(\mathcal N)$. This completes our proof.
 \hfill $\Box$
\medskip

%
%

Every filter $\F$ on $\w$ gives rise to the filter $\F^{(<\w)}$ on
$\mathit{Fin}:=[\w]^{<\w}\setminus\{\emptyset\}$ generated by sets
$[F]^{<\w}\setminus\{\emptyset\}$, where $F\in\F$. For a family
$\mathcal B$ of infinite subsets of a countable set $X$ we denote by
$\B^+$ the family $\{Z\subset X:\forall B\in\B (|Z\cap B|=\w)\}$.
For every $E\subset \mathit{Fin}$ let us denote by $\K(E)$ the
family $\{K\subset\w:\forall e\in E(e\cap K\neq\emptyset)\}$. It is
easy to see that $\K(E)$ is always compact and $\K(E)\subset
[\w]^\w$ if for every $n\in\w$ there exists $e\in E$ such that $\min
e>n$. It is a straightforward exercise to check that $E\in
(\F^{(<\w)})^+$ iff $\K(E)\subset\F^+$.

In the next proof, we will use the notation $\zrost$ for the set of
the increasing functions from $\omega$ to $\omega$. Also, we will use the
fact that $\bb = \omega_2$ holds in the Laver model.

\medskip

 \noindent\textit{Proof of Theorem~\ref{laver}.} \ Let $\F=\F(\A)$. By
\cite[Corollary~2.2]{ChoRepZdo15} it suffices to prove that for
every decreasing sequence $\la S_n:n\in\w\ra$ of elements of
$(\F^{(<\w)})^+$ there exists $f\in\w^\w$ such that
$S_f:=\bigcup_{n\in\w}(S_n\cap\mathcal P(f(n)))$ belongs to
$(\F^{(<\w)})^+$, i.e., $\K(S_f)\subset\F^+$. Without loss of
generality we may assume that $\min s>n$ for all $s\in S_n$.

Since $\A$ is $\w$-mad, for every   countable family
$$\{\la
X^i_n:n\in\w\ra:i\in\w\}\subset\prod_{n\in\w}\K(S_n)$$
 there exists
$A\in\A$ such that $|A\cap X^i_n|=\w$ for all $i,n\in\w$. We claim
that there are actually $\w_2$-many $A\in\A$ as above. Indeed,
suppose that for some $\A'\in [\A]^{\w_1}$ there is no
$A\in\A\setminus\A'$ such that $|A\cap X^i_n|=\w$ for all
$i,n\in\w$. Fix a sequence $\la A_n:n\in\w\ra$ of mutually different
elements of $\A\setminus\A'$ and find $h\in\zrost$ such that
$$ \la\max (A\cap A_n)+1:n\in\w\ra \leq^* h$$
for all $A\in\A'$. Such an $h$ exists because $|\A'|<\hot b=\w_2$.
Set $X=\bigcup_{n\in\w}(A_n\setminus h(n))$ and note that $X\in\F^+$
and $|X\cap A|<\w$ for all $A\in\A' $. It follows that there is no
$A\in\A$ which intersects infinitely often all elements of the
family $\{X^i_n:i,n\in\w\}\cup \{X\}$, a contradiction.

 Let
$f\in\w^\w$ be increasing and such that $f(n)>\min(A\cap X^i_n)$ for
all $i\leq n$. Thus for every $i$ and all $n\geq i$ we have $A\cap
X^i_n\cap
 f(n)\neq\emptyset$. Set
$$G_{A,f}=\big\{\la X_n:n\in\w\ra\in \prod_{n\in\w}\K(S_n):
\exists^\infty n(A\cap X_n\cap  f(n)\neq\emptyset)\big\}$$ and note
that $G_{A,f}$ is a $G_\delta$-subset of $\prod_{n\in\w}\K(S_n)$
containing $\la X^i_n:n\in\w\ra$ for all $i\in\w$. Thus we have
proven that for every countable $Q\subset \prod_{n\in\w}\K(S_n)$
there exists $A\in\A$ and $f\in\zrost$ such that $Q\subset G_{A,f}$.
Moreover, there are $\w_2$-many such pairs $\la A,f\ra$ with
mutually different first coordinates. Let us fix $\A'\in
[\A]^{\w_1}$. Applying \cite[Lemma~2.2]{RepZdo17} we conclude that
there exists a family $\{\la
A_\alpha,f_\alpha\ra:\alpha<\w_1\}\subset\A\times\zrost$ such that
$\prod_{n\in\w}\K(S_n)\subset\bigcup_{\alpha<\w_1}G_{A_\alpha,f_\alpha}$
and $\A'\cap\{A_\alpha:\alpha<\w_1\}=\emptyset$. Since $\A'$ was
chosen arbitrarily, it follows from the above that we can
additionally assume that each $\la X_n:n\in\w\ra\in
\prod_{n\in\w}\K(S_n)$ is contained in $G_{A_\alpha, f_\alpha}$ for
infinitely many $\alpha$. Pick $f\in\zrost$ such that
$f_\alpha\leq^*f$ for all $\alpha$. We claim that
$\K(S_f)\subset\F^+$. Indeed, for every $n\in\w$ and $s\in
S_n\cap\mathcal P(f(n))$ select $k_{s,n}\in s$. We are left with the
task to prove that $X=\{k_{s,n}:s\in S_n\cap\mathcal
P(f(n))\}\in\F^+$. For this sake,  for every $n$ and $s\in
S_n\setminus\mathcal P(f(n))$ select $l_{s,n}\in s\setminus f(n)$
and consider the sequence $\la X_n:n\in\w\ra\in
\prod_{n\in\w}\K(S_n)$, where
$$ X_n = \big\{k_{s,n}: s\in S_n\cap\mathcal
P(f(n))\big\} \cup \big\{ l_{s,n}: s\in S_n\setminus\mathcal
P(f(n))\big\}.
$$
Our proof will be completed as soon as we show that $X\cap A_\alpha$
is infinite for all $\alpha$ such that $\la X_n:n\in\w\ra\in
G_{A_\alpha,f_\alpha}$. So let us fix such an $\alpha$ and
$m_0\in\w$. Let $m\geq m_0$ be such that $f_\alpha(n)\leq f(n)$ for
all $n\geq m$.
 By the definition of $G_{A_\alpha,f_\alpha}$ there exists
 $n\geq m$ such that $\emptyset\neq X_n\cap A_\alpha\cap
 f_\alpha(n)$,
 and hence $\emptyset\neq X_n\cap A_\alpha\cap
 f(n)$. Fix $j$ in the latter intersection.
 It follows that $j$ cannot be of the form $ l_{s,n}$
 for $s\in S_n\setminus\mathcal
P(f(n))$ because $j\in f(n)$, an hence $j=k_{s,n}$ for some $s\in
S_n\cap\mathcal P(f(n))$, which yields $j\in X$ and thus completes
our proof. \hfill $\Box$
\medskip

\noindent\textbf{Acknowledgements.} The
 work reported here
was carried out during the visit of the second named author at the
Instituto de Ci\^encias Matem\'aticas e de Computa\c c\~ao, Universidade de S\~ao Paulo, in July 2018.  This visit
was  supported  by  FAPESP (2017/09252-3). The second named author thanks the first
named author and Lucia Junqueira  for their kind hospitality. In addition, we thank Osvaldo Guzman
for his comments on the previous version of this paper.

We are also grateful to the anonymous referee who among other things
pointed out that our proof of Theorem~\ref{p_equals_c} requires also
the assumption  $\mathit{cov}(\mathcal N)=\hot c$.

\end{document}